\documentclass[a4paper,10pt]{amsart}

 \usepackage{amsaddr}
\usepackage[utf8]{inputenc}
\usepackage[dvips]{epsfig}
\usepackage{color}
\usepackage{moreverb}
\usepackage{amsfonts}
\usepackage{amsmath}
\usepackage{fullpage}
\usepackage{algorithm}
\usepackage[noend]{algorithmic}

\newcommand{\conv}{\mathrm{conv}}
\newcommand{\iso}{\mathrm{iso}}
\newcommand{\Ima}{\mathrm{Im}}
\newcommand{\dist}{\mbox {dist}}
\theoremstyle{plain}
\newtheorem{thm}{Theorem}[section]
\newtheorem{cor}[thm]{Corollary}
\newtheorem{prop}[thm]{Proposition}
\newtheorem{lem}[thm]{Lemma}

\newtheorem{quest}[thm]{Question}
\newtheorem{claim}{Claim}

\newenvironment{proofclaim}{\noindent{\em Proof of the claim.}}{\qedclaim}

\newcommand{\qedclaim}{\hfill $\diamond$ \medskip}
\title{Computing metric hulls in graphs}
\author{Kolja Knauer}
\address{Aix Marseille Univ, Universit\'e de Toulon, CNRS, LIS, Marseille, France}
\email{kolja.knauer@lis-lab.fr}

\author{Nicolas Nisse}
\address{Universit\'e C\^ote d'Azur, Inria, CNRS, I3S, France}
\email{nicolas.nisse@inria.fr}
\thanks{NN was supported by ANR program ``Investments for the Future'' under
		reference ANR-11-LABX-0031-01 and the associated Inria team AlDyNet. KK was supported by grants DISTANCIA: ANR-17-CE40-0015 and GATO: ANR-16-CE40-0009-01.}

\begin{document}

\begin{abstract}
We prove that, given a closure function the smallest preimage of a closed set can be calculated in polynomial time in the number of closed sets. %This confirms a conjecture of Albenque and Knauer and 
This implies that there is a polynomial time algorithm to compute the convex hull-number of a graph, when all its convex subgraphs are given as input. We then show that computing if the smallest preimage of a closed set is logarithmic in the size of the ground set is LOGSNP-complete if only the ground set is given. A special instance of this problem is computing the dimension of a poset given its linear extension graph, that was conjectured to be in P.

The intent to show that the latter problem is LOGSNP-complete leads to several interesting questions and to the definition of the isometric hull, i.e., a smallest isometric subgraph containing a given set of vertices $S$. While for $|S|=2$ an isometric hull is just a shortest path, we show that computing the isometric hull of a set of vertices is NP-complete even if $|S|=3$. Finally, we consider the problem of computing the isometric hull-number of a graph and show that computing it is $\Sigma^P_2$ complete.
 \end{abstract}

\maketitle
 
\section{Introduction}
We study the complexity of several algorithmic problems arising from metric hulls in graphs. Let $G=(V,E)$ be a graph. A set $S \subseteq V$ is {\it convex} if, for any $u,v \in S$, any $(u,v)$-shortest path in $G$  is included in $S$. The {\it convex hull} $\conv(S)$ of a set $S \subseteq V$ is the smallest convex set containing $S$. A {\it (convex) hull set} of $G$ is a set $S \subseteq V$ such that $\conv(S)=V$. The {\it hull-number} $hn(G)$ of $G$ is the size of a minimum hull set of $G$. Let $S \subseteq V$ be a convex set of $G$, let $hn_G(S)$ denote the size of a minimum set $Z \subseteq V$ such that $S=\conv(Z)$. The hull-number was introduced in~\cite{Eve-85}, and since then has been the object of numerous papers, in particular because this may model contamination
spreading processes. Most of the results on the hull number are about computing good bounds for specific graph classes, see
e.g.~\cite{Cha-00,Her-05,Can-06,Cac-10,Dou-10,Cen-13}. 

While computing the convex hull of a set vertices can be done in polynomial time, computing the hull-number of a graph $G$ is known to be NP-complete~\cite{Dou-09} and remains so even if $G$ is a bipartite graph~\cite{Ara-13} or moreover a partial cube~\cite{Alb-16}, i.e., an isometric subgraph of a hypercube. It is well-known, that the function $\conv:2^V\to 2^V$ is a \emph{closure}, where generally a function $cl:2^A\to 2^A$, is called closure if it satisfies the following conditions for all sets $X,Y\subseteq A$
\begin{itemize}
 \item $X\subseteq {cl}(X)$ \hfill (\emph{extensive})
 \item $X\subseteq Y\implies {cl}(X)\subseteq {cl}(Y)$ \hfill (\emph{increasing})
 \item ${cl}({cl}(X))={cl}(X)$ \hfill (\emph{idempotent})
\end{itemize}
 
This leads to the following generalization of the hull-number of a graph, that we call Minimum Generator Set ({\bf \texttt{MGS}}):

\begin{quote}
Given a set $A$, a polytime computable closure $cl:2^A\to2^A$ and an integer $k$,
is there a set $X \subseteq A$ with $|X|\leq k$ such that $cl(X)=A$?
\end{quote}

Together with the results about the hull-number it follows that \texttt{MGS} is NP-complete. However, in~\cite{Alb-16} it was conjectured, that \texttt{MGS} is solvable in polynomial time if the set of images of $cl$ is part of the input. Our first result is to prove this conjecture under the weaker assumption of having a \emph{pseudo-closure}, which is what we call $f$ if it only satisfies $f(X \cup Y)=f(f(X) \cup f(Y))$ for any $X,Y \subseteq A$. More precisely, we devise an algorithm that finds minimum generating sets for all images of $f$ in polynomial time (in $|A|$, $|\Ima(f)|$ and the computation time of $f$)  (Theorem~\ref{thm:poly}). 
%\textcolor{red}{do we need that $f$ is polytime computable, if its images are given as an input? for a closure it is not necessary.} \nico{It seems that it is required, on Line 9 of the algorithm. Actually, with the time complexity that we mention, it should be computable in constant time?} \kolja{I agree, maybe we should include the compelxity of computing $f$ into the formula?}
 Pseudo-closures embody a large class of objects. In particular, closures are essentially the same as lattices (see the discussion above Question~\ref{quest:graphiclattices}). Thus, deciding whether for an element $\ell$ in a lattice $L$ there are $k$ join-irreducibles, whose join is $\ell$, can be done in polynomial time. Lattices encode many combinatorial objects (for the cases below see~\cite{Her-94,She-96}). Thus, our results for instance yield polynomial time algorithms to decide whether:
\begin{itemize}
 \item a finite metric space has a hull-set of size $k$, if all convex sets are given as input,
 %\item a matroid has a basis of size at most $k$, if all flats are given,
 %\item the ground set in a discrete convex geometry is the hull of at most $k$ elements, if all convex sets are given,
 \item a (semi)group can be generated by $k$ elements, if all sub(semi)groups are given,
\end{itemize}
%\kolja{j'ai enleve les autres deux exemples, parce-que dans ces cas il est facil et connu de faire en temps poly sans tout ce input}
%
We then consider \texttt{MGS} with input (only) $A$ with respect to \emph{atomistic} closures, i.e.,  $\conv(\{x\})=\{x\}$ for all $x\in A$, and show that it is W[2]-complete (Corollary~\ref{cor:W2}) and its log-variant \texttt{LOGMGS} 
\begin{quote}
Given a set $A$, a polytime computable atomistic closure $cl:2^A\to2^A$ and an integer $k\leq \log(|A|)$,
is there a set $X \subseteq A$ with $|X|\leq k$ such that $cl(X)=A$?
\end{quote}
is LOGSNP-complete (Cor.~\ref{cor:LOGSNP})~\cite{Pap-96}.%, a class introduced by Papadimitriou and Yannakakis~\cite{Pap-96}.

Whether the corresponding results hold for the hull-number of a graph is open. In particular, in~\cite{Alb-16} it was conjectured that given a poset together with all its linear extensions, its dimension can be computed in polynomial time. Moreover, it was shown this problem is an instance of \texttt{LOGHULL-NUMBER} for partial cubes. While trying to establish a reduction to the hull-number problem for partial cubes, in order to show that \texttt{LOGHULL-NUMBER} for partial cubes is LOGSNP-complete, we arrive at the second metric hull-problem of concern in this paper, which to our knowledge is a novel problem: Let again $G=(V,E)$ be a graph. A set $S \subseteq V$ is {\it isometric} if, for any $u,v \in S$, a $(u,v)$-shortest path if $G$ is included in $S$. An {\it isometric hull} $\iso(S)$ of a set $S \subseteq V$ is a smallest isometric set containing $S$. Thus, this problem can be seen as a variation of the Steiner Tree problems. In particular, an isometric hull of two vertices is simply a shortest path. An {\it isometric hull set} of $G$ is a set $S \subseteq V$ such that $\iso(S)=V$. The {\it isometric hull-number} $ihn(G)$ of $G$ is the size of a minimum isometric hull set of $G$. In Section~\ref{sec:iso} we show that computing an isometric hull of a set of vertices $S$ is NP-complete even if $|S|=3$ (Theorem~\ref{theo:minIsomHull}) and that computing the isometric hull-number of a graph is $\Sigma^P_2$ complete (Theorem~\ref{theo:minHullSet}).

\section{Minimum generators of pseudo-closures}

Let $A$ be any set and $f: 2^A \rightarrow 2^A$ a \emph{pseudo-closure}. Note, that setting $Y=X$ in $f(X \cup Y) = f(f(X) \cup f(Y))$ one obtains, that $f(X)=f(f(X))$ for all $X\subseteq A$, i.e., pseudo-closures are idempotent. In particular, pseudo-closures generalize closures in a different way than \emph{preclosures}, which are not required to be idempotent. 
Finally, $f$ is said {\it size-increasing} if $X\subseteq Y \implies |f(X)|<|f(Y)|$ or $f(X)=f(Y)$ for all $X,Y\subseteq A$.
Let us first argue that in a way pseudo-closures are closures without the property of being extensive.

\begin{lem}\label{lem:closure}
 Let $f: 2^A \rightarrow 2^A$. % \nico{\sout{be a pseudo-closure}}.
  The following are equivalent:
 \begin{itemize}
  \item[(i)] $f$ is a closure,
  \item[(ii)] $f$ is an extensive pseudo-closure, 
  \item[(iii)] $f$ is an extensive and size-increasing pseudo-closure.
 \end{itemize}
\end{lem}
\begin{proof}
(i)$\implies$(iii): Let $f$ be a closure, then it is extensive and increasing by definition and clearly also size-increasing. Since $f(X) \subseteq f(X \cup Y)$ and $f(Y) \subseteq f(X \cup Y)$, then $f(X) \cup f(Y) \subseteq f(X \cup Y)$. Hence, $f(f(X) \cup f(Y)) \subseteq f(f(X \cup Y))=f(X \cup Y)$. 
Moreover, $X \subseteq f(X)$ and $Y \subseteq f(Y)$, therefore, $f(X \cup Y) \subseteq f(f(X) \cup f(Y))$. Therefore $f$ is an extensive and size-increasing pseudo-closure. 

(iii)$\implies$(ii): trivial.

(ii)$\implies$(i): Let $f$ be an extensive pseudo-closure. As argued above $f$ is idempotent. It remains to prove that $f$ is increasing. Let $X\subseteq Y$. We have $$f(X)\subseteq f(X)\cup f(Y)\subseteq f(f(X) \cup f(Y))=f(X \cup Y)=f(Y),$$ where the second inclusion uses that $f$ is extensive and the first equality uses that $f$ is a pseudo-closure.
%\qed

\end{proof}

% \textcolor{red}{the following is probably not interesting enough (unless we find some nice application of it), so I will probably take it out and prove the proposition below on its own.} 
% \begin{lem}\label{lem:}
%  Let $f: 2^A \rightarrow 2^A$ be a pseudo-closure and $i,j: 2^A \rightarrow 2^A$ be such that $i(X\cup Y)=i(X)\cup i(Y)$, $j(X\cup Y)=j(X)\cup j(Y)$, and $i\circ j\circ f\circ i(X)=f\circ i(X)$ or $j\circ f\circ i \circ j(X)=j\circ f(X)$ for all $X,Y\subseteq A$. Then $j\circ f\circ i$ is a pseudo-closure.
% \end{lem}

Just to give an example of a pseudo-closure, that is not a closure, consider:
\begin{prop}\label{prop:xmpl}
 Let $cl: 2^A \rightarrow 2^A$ be a closure and $\emptyset\neq X'\subseteq X\subseteq A$. Then $f(Y):=cl(Y\cup X)\setminus X'$ is an increasing pseudo-closure that is not extensive.
\end{prop}
\begin{proof}
 To see that $f$ is a pseudo-closure, we transform $f(f(Y)\cup f(Z))=cl(cl(Y\cup X)\setminus X'\cup cl(Z\cup X)\setminus X'\cup X)\setminus X'$ which by $X'\subseteq X$ equals $cl(cl(Y\cup X)\cup cl(Z\cup X)\cup X)\setminus X'$ which since $cl$ is extensive equals $cl(cl(Y\cup X)\cup cl(Z\cup X))\setminus X'$. Now, since by Lemma~\ref{lem:closure} $cl$ is a pseudo-closure, we can transform to  $cl((Y\cup X)\cup (Z\cup X))\setminus X'$ which equals $cl(Y\cup Z \cup X)\setminus X'=f(Y\cup Z)$.
 
 It is easy to see, that $f$ is increasing and since $X'\not\subseteq f(X')$ it is not extensive.
\end{proof}

%\textcolor{red}{still find a non-increasing but size-increasing and a not size-increasing pseudo-closure?}

We now turn our attention to the problem of generating images of a pseudo-closure.
A set $X \subseteq A$ {\it generates} $f(X)$, and $X$ is {\it minimum} (for $f$)  if there is no set $Y \subseteq A$ such that $|Y|<|X|$ and $f(X)=f(Y)$. 

\subsection{Generating with large input}
In this section, we design a dynamic programming algorithm that computes a minimum generator of any $H \in \Ima(f)=\{Y \subseteq A\mid \exists X \subseteq A, Y=f(X)\}$. {We assume that, for any $X' = X \cup \{w\} \subseteq A$ and  given $f(X)$,  determining $f(X')$ can be done in time $c_f$}. A similar algorithm has been published previously in a different language and restricted to closure functions~\cite{Kam-05}. Moreover, the approach in~\cite{Kam-05} is incremental which leads to a time-complexity of $O(c_f|A||\Ima(f)|^2)$ (while no runtime analysis is presented there). We include our algorithm here to be self-contained but also because the complexity $O(c_f|A||\Ima(f)|)$ of our algorithm is slightly better and we think that our presentation might be more accessible to our community.

Let us describe the algorithm informally. Every set $S \in \Ima(f)$ is assigned to one of its generators stored in the variable $label(S)$. Initially, $label(S)$ may be any generator of $S$ (for instance, $S$ itself). The algorithm considers the sets in $\Ima(f)$ in non decreasing order of their size and aims at refining their labels. More precisely, from a set $Y \in \Ima(f)$ with generator $label(Y)$, the algorithm considers every set $f(R)$ generated by $R=label(Y) \cup \{z\}$ for some $z \in A$. If $R$ is smaller than $label(f(R))$ then $R$ becomes the new label of $f(R)$. 

\begin{algorithm}[ht]
  \caption{$MinGen(A,f)$.}\label{alg:2}

  \begin{algorithmic}[1]
    % \footnotesize
    \REQUIRE{A set $A$, a pseudo-closure  $f: 2^A \rightarrow 2^A$, and the set $\Ima(f)$.}
    \STATE For any $H \in \Ima(f)\setminus\{f(\emptyset)\}$, set $label(H) \leftarrow H$ and set $label(f(\emptyset)) \leftarrow \emptyset$
    \STATE Set $Continue \leftarrow True$
    \WHILE{$Continue$}
        \STATE Set $Continue \leftarrow False$
    \FOR{$i=1$ to $|A|$}
%        \FOR{$x \in L$, $|x|=i-1$}
%        \STATE $mark(x) \leftarrow label(x)$.
%        \ENDFOR
    \FOR{$Y \in \Ima(f)$, $|Y|=i$}
    \FOR{$z \in A \setminus label(Y)$}
    \STATE Set $R \leftarrow \{z\} \cup label(Y)$
    \STATE Set $H \leftarrow label(f(R))$

    \IF{$|R| < |H|$} \STATE { $label(f(R)) \leftarrow  R$ and $Continue \leftarrow True$}
    \ENDIF
            \ENDFOR
        \ENDFOR
    \ENDFOR
    \ENDWHILE
    \RETURN $\{label(Y) \mid Y \in \Ima(f)\}$
  \end{algorithmic}
\end{algorithm}

%\vspace{-0.5cm}

\begin{thm}\label{thm:poly}
Algorithm $MinGen(A,f)$ computes a minimum generator of any $H \in \Ima(f)$ in time $O(c_f(|A| |\Ima(f)|)^2)$. 

Moreover, if $f$ is size-increasing, its time-complexity is $O(c_f|A||\Ima(f)|)$.
\end{thm}

\begin{proof}
Let us first show that, at the end of the execution of the algorithm, $label(Y)$ is a minimum generator for every $Y \in \Ima(f)$. 

Clearly, $label(Y)$ is initially  a generator of $Y$ (Line $1$). Moreover, $label(Y)$ can only be modified when it is replaced by $R$ such that $f(R)=Y$ (Line 11). Let us show that $label(Y)$ is minimum. 
 
For purpose of contradiction, let $Y \in \Ima(f)$ such that the value $L$ of $label(Y)$ at the end of the algorithm is not a minimum generator of $Y$. Moreover, let us consider such a counter example such that the size of a minimum generator is minimum. Hence, there is $Z \subseteq A$ with $|Z|<|L|$ and $f(Z)=f(L)=Y$ and $Z$ is a minimum generator for $Y$. By line 1, we know that $Z\neq \emptyset$. Hence, let $w \in Z$ and $X=f(Z \setminus \{w\})$.  Any minimum generator of $X$ has size at most $|Z|-1$. 
%\nico{Should we consider the case when $|Z|=1$? i.e., is it possible that $f(\emptyset)=X$ ?} \kolja{It is possible, but I dont see a problem with it. The case $Z=\emptyset$ had to be dealt with separately, though.}
 Therefore, by minimality of the size of a minimum generator of our counter-example, $label(X)$ is a minimum generator of $X$. In particular, $f(label(X))=X=f(Z \setminus \{w\})$.

First, let us show that $w \notin label(X)$. Indeed, otherwise, $X=f(label(X))=f(label(X) \cup \{w\})=f(f(label(X)) \cup f(w))=f(f(Z\setminus \{w\}) \cup f(w))=f(Z)=Y$. Therefore, $f(Z\setminus \{w\}) = X = Y$, contradicting the fact that $Z$ is a minimum generator for $Y$.

Consider the step when $label(X)$ receives its final value. After this step, $Continue$ must equal $True$. Therefore, there is another iteration of the $While$-loop. During this next iteration, there must be an iteration of the $For$-loop (Line 6) that considers $X \in \Ima(f)$ and an iteration of the $For$-loop (Line 7) that consider $w \notin label(X)$. At this iteration, we set $H=f(label(X) \cup \{w\})= f(f(label(X)) \cup f(w))=f(f(Z\setminus \{w\}) \cup f(w))=f(Z)=Y$. Because the size of the set $label(Y)$ is non increasing during the execution, the value $L'$ of $label(Y)$ at this step is such that $|L|\leq |L'|$. In particular, $|label(X) \cup \{w\}| \leq |Z| <|L| \leq |L'|$. Therefore, during this execution (Line 11), $label(Y)$ should become equal to $label(X) \cup \{w\}$. Since, again, the size of the set $label(Y)$ is non increasing, it contradicts the fact that $label(Y)=L$ at the end of the algorithm. 

%\medskip

First note, that since $f$ is idempotent and $H\in \Ima(f)$ in Line 1 we can set $label(H) \leftarrow H$, i.e., this can be done in constant time.
Each iteration of the $While$-loop takes time $O(c_f|A||\Ima(f)|)$. Moreover, each new iteration of this loop comes after a modification of some label in the previous iteration (Line 11, because $Continue$ is set to $True$). Since there are $|\Ima(f)|$ labels and each of them will receive at most $|A|$ values (because the size of a label is not increasing), the time-complexity of the algorithm is $O(c_f(|A| |\Ima(f)|)^2)$.

%\medskip

In case when $f$ is size-increasing, we prove that each label contains its final value after the first iteration of the $While$-loop. So, there is exactly $2$ iterations of this loop in that case and the time-complexity is $O(c_f|A||\Ima(f)|)$ when $f$ is size-increasing.

More precisely, we show that the label of $Y \in \Ima(f)$ contains its final value just before $Y$ is considered in the $For$-loop (line 6) of the first iteration. The proof is similar to the one of the correctness of the algorithm. 

For purpose of contradiction, let $Y \in \Ima(f)$ such that the value $L$ of $label(Y)$ just before $Y$ is considered in the $For$-loop (line 6) of the first iteration is not a minimum generator of $Y$. Moreover, let us consider such a counter example such that $|Y|$ is minimum. Hence, there is $Z \subseteq A$ with $|Z|<|L|$ and $f(Z)=f(L)=Y$. Let $w \in Z$ and $X=f(Z \setminus \{w\})$.  Any minimum generator of $X$ has size at most $|Z|-1$. Moreover, because $f$ is size-increasing, $|X|=|f(Z\setminus \{w\})| < |f(Z)|=|Y|$ (because $X \neq Y$ since their minimum generators have different sizes). Therefore, by minimality of the counter-example, $label(X)$ is a minimum generator of $X$ just before $X$ is considered in the $For$-loop of the first iteration, and moreover, $X$ is considered before $Y$. In particular, $f(label(X))=X=f(Z \setminus \{w\})$.

Similarly as before, $w \notin label(X)$. Hence, during the iteration (of the $For$-loops) that considers $X$ and $w$, either $label(Y)$ must become $label(X) \cup \{w\}$ or $|label(Y)|\leq |label(X) \cup \{w\}|\leq |Z|$. In both cases, it is a contradiction since $X$ is considered before $Y$ and $ |label(X) \cup \{w\}| < |L|$.%\qed
\end{proof}

%\vspace{-0.2cm}

% 
% 
% It is well-known that closures correspond to lattices in the following way: Given a closure $cl$ define the inclusion order on the closed sets, i.e.,  $\Ima(cl)$. Since this order has unique maximal element $A$ and the intersection of closed sets is closed, it is a lattice. On the other hand given a lattice $L$ with set $J$ of join-irreducibles associate to every $\ell\in L$ the set $J_{\ell}$ of join-irreducibles that are less or equal than $\ell$. Now, $L$ is the inclusion order on $\{J_{\ell}\mid \ell\in L\}$. In turn the latter is the set of closed sets of $cl:2^J \rightarrow 2^J$ defined as $X\mapsto J_{\bigvee X}$, which is easily seen to be a closure. 
% Recall that a lattice is \emph{atomistic} if all its elements can be generated as joins of atoms. Here is the corresponding class of closures: a closure $cl:2^A \rightarrow 2^A$ is \emph{atomistic} if $cl(\{x\})=\{x\}$ for all $x\in A$. 

From Th.~\ref{thm:poly}, Lemma~\ref{lem:closure}, and the preceding discussion we immediately get:

\begin{cor}\label{cor:closure}
 Let $cl:2^A\to 2^A$ be a closure. \texttt{MGS} can be solved in $O(c_{cl}|A||\Ima(cl)|)$ time.
\end{cor}

\noindent This confirms a conjecture of~\cite{Alb-16} (which could have probably also been extracted from~\cite{Kam-05}) and slightly improves the time-complexity of~\cite{Kam-05} in the case of a closure. Furthermore, it is well-known and easy to see that for a closure $cl(X)=\bigcap_{X\subseteq Y\in\Ima(cl)}Y$. Thus, $c_{cl}$ is in $O(|\Ima(cl)|)$ yielding a uniform bound of $O(|A||\Ima(cl)|^2)$.

%We have presented a polynomial time algorithm for \texttt{MGS} for pseudo-closures. However, we are not aware of natural pseudo-closures, that are not also closures. 

%\vspace{-0.3cm}

\subsection{Generating with small input}

%\vspace{-0.2cm}

In this section we show that for an atomistic closure $cl:2^A\to 2^A$, the problem \texttt{MGS} is W[2]-complete with respect to the size of the solution, when only $A$ is the input. Furthermore, \texttt{LOGMGS} is LOGSNP-complete. We then introduce the problem \texttt{COORDINATE REVERSAL}, show an equivalence with \texttt{HITTING SET}, and finally relate it to the hull-number problem in partial cubes.

% % Indeed, we can furhtermore restrict the instances of the probelm to complemented lattices. A lattice is called \emph{complemented} if for every $\ell\in L$ there is a $\bar{\ell}$ such that $\ell\wedge \bar{\ell}=0$ and $\ell\vee \bar{\ell}=1$.
% 
% We consider the following problem Minimum Generator Set ({\bf \texttt{MGS}}):
% \begin{quote}
% Given an atomistic lattice $L$ with set of atoms $A$, encoded by a polytime computable closure $cl:2^A\to2^A$ and an integer $k$,
% is there a set $X \subseteq A$ of atoms with $|X|\leq k$ such that $cl(X)=\max(L)$.
% \end{quote}

% 
% 
% We recall the combinatorial optimization problem \texttt{MGS} in the closure formulation: determine the minimum $k$, such that there is $X\subseteq A$ with $|X|=k$ and $cl(X)=A$. 
For us, the optimization problem \texttt{DOMINATING SET} consists in given a directed graph $D=(V,A)$ to find the minimum $k$, such that there is $X\subseteq V$ with $|X|=k$ and every $v\in V\setminus X$ has an incoming arc from a vertex in $X$.

\begin{prop}\label{prop:eqdom}
 \texttt{MGS} and \texttt{DOMINATING SET} are L-equivalent, i.e., there are polynomial reductions both ways that preserve optimal solutions.
\end{prop}
\begin{proof}
 Given a directed graph $D=(V,A)$ we construct an atomistic closure $cl$ such that a minimum dominating set of $D$ is of the same size as a minimum generating set of $cl$. 
 First, we can assume that the collection of closed in-neighborhoods $\{\overline{N}^-(v)\mid v\in V\}$ is a set without inclusions, since dominating set is L-equivalent to hitting set on $(V,\{\overline{N}^-(v)\mid v\in V\})$. 
 For the same reason, we may assume that for any two vertices $u,v\in V$ there is a set in closed in-neighborhood containing $u$ but not $v$. We define $cl:2^V \rightarrow 2^V$ by setting the closed sets to be all possible intersections of the complements of closed in-neighborhoods $\{V\setminus\overline{N}^-(v)\mid v\in V\}$, i.e., $cl(X)$ is the smallest set of vertices containing $X$ which can be expressed as intersection of $\{V\setminus\overline{N}^-(v)\mid v\in V\}$. It is easy to check that $cl$ is a closure and since for any two vertices $u,v\in V$ there is a set in closed in-neighborhood containing $u$ but not $v$, we have $cl(\{v\})=\{v\}$ for all $v\in V$, i.e., $cl$ is atomistic. Clearly, $cl(X)$ can be computed in polynomial time for a given $X\subseteq V$.
 
 Now, a dominating set $X\subseteq V$ corresponds to a hitting set of $\{\overline{N}^-(v)\mid v\in V\}$ which in turn is a set of elements not contained in any of the maximal proper closed sets of $cl$, i.e., $cl(X)=V$. This construction is clearly reversible, so we have proved the claim.%\qed
\end{proof}

While it was known before that determining the hull-number of graphs, bipartite graphs, and even partial cubes is NP-complete, the known reductions reduce variants of 3-SAT to the decision version of hull-number. Here, we have shown that two decision versions of combinatorial optimization problems are equivalent in the stronger sense of L-reductions, i.e., sizes of solutions are preserved. This has some immediate consequences. 

Since \texttt{\texttt{DOMINATING SET}} is W[2]-complete, Proposition~\ref{prop:eqdom} gives:
\begin{cor}\label{cor:W2}
 \texttt{MGS} is W[2]-complete.
\end{cor}

Using results of Dinur and Steurer~\cite{Din-14}, by Proposition~\ref{prop:eqdom} we get:
\begin{cor}\label{cor:INAPPROX}
 \texttt{MGS} cannot be approximated to $\bigl(1 - o(1)\bigr) \cdot \ln{n}$ unless P=NP.
\end{cor}

% In~\cite[Theorem 15.59]{Flu-06} it is shown that {\texttt{LOGDOMINATING SET}} is LOGSNP (aka LOG[2]) complete, which with Proposition~\ref{prop:eqhit} gives 

In~\cite{Pap-96} it is shown that {\texttt{LOGDOMINATING SET}} is LOGSNP (aka LOG[2]) complete, which with Proposition~\ref{prop:eqdom} gives:

\begin{cor}\label{cor:LOGSNP}
 \texttt{LOGMGS} is LOGSNP-complete.
\end{cor}

Since the $\conv$-operator for graphs is an atomistic closure, we wonder if similar results can be proved for the hull-number problem, or if this problem is essentially easier. For instance in~\cite{Ara-13b}, a fixed parameter tractable algorithm to compute the hull-number of any graph was obtained. But there the parameter is the size of a vertex cover. How about the complexity when parameterized by the size of a solution?

It is well-known that closures correspond to lattices in the following way: Given a closure $cl$ define the inclusion order on the closed sets, i.e.,  $\Ima(cl)$. Since this order has a unique maximal element $A$ and the intersection of closed sets is closed, it is a lattice. On the other hand given a lattice $L$ with set $J$ of join-irreducibles associate to every $\ell\in L$ the set $J_{\ell}$ of join-irreducibles that are less or equal than $\ell$. Now, $L$ is the inclusion order on $\{J_{\ell}\mid \ell\in L\}$. In turn the latter is the set of closed sets of $cl:2^J \rightarrow 2^J$ defined as $X\mapsto J_{\bigvee X}$, where $\bigvee X$ denotes the join of $X$. This is easily seen to be a closure. 
Recall that a lattice is \emph{atomistic} if all its elements can be generated as joins of atoms. The corresponding class of closures are precisely the atomistic closures. Now we can state the following :

\begin{quest}\label{quest:graphiclattices}
 What are the atomistic lattices that come from the convex subgraphs of a graph?
\end{quest}

Clearly, these lattices are quite special, in particular any such lattice is entirely determined by its first two levels, since these correspond to vertices and edges of the graph. On the other hand, it is not clear what other properties such lattices enjoy. For instance, the graph in Figure~\ref{fig:graphworank} shows that convexity lattices of graphs are not \emph{graded} in general, i.e., not all maximal chains are of the same length.

\begin{figure}[htb]
\begin{center}
\includegraphics{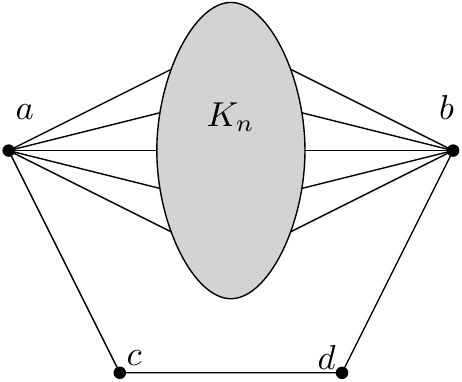}
\end{center}
\caption{\label{fig:graphworank} A family of graphs with lattice of convex subgraphs being arbitrary far from ranked.}
\end{figure} 

For example, \emph{Ptolemaic} graphs are exactly those graphs whose lattice of convex subgraphs is lower locally distributive~\cite{Far-86}.
Also, in~\cite{Alb-16} the lattices of convex subgraphs of partial cubes were characterized.  
However, we do not know how to make use of this characterization. 

Let us now approach the hull-number problem in partial cubes via a different reduction.
We call {\texttt{COORDINATE REVERSAL}} the following problem: \begin{quote}
 Given a set $X$ of vertices of the hypercube $Q_d$ and an integer $k$, is there a subset $X'\subseteq X$, with $|X'|\leq k$ and such that for every coordinate $e$ of $Q_d$ there are vertices $x,y\in X'$ with $x_e\neq y_e$.
                                                          
                                                          \end{quote}
 
\begin{prop}\label{prop:COORDINATE}
 \texttt{COORDINATE REVERSAL} is L-equivalent to \texttt{HITTING SET}.
\end{prop}
\begin{proof}
Let $(V,\mathcal{S})$ be an instance of hitting set with $\mathcal{S}=\{S_1, \ldots, S_k\}$, where $S_k=V$. This assumption clearly does not change the problem.
Assume that for any two vertices $u,v\in V$ there is a set in $\mathcal{S}$ containing $u$ but not $v$, which is clearly OK as an assumption. Second, we extend $V$ with one vertex $x$ and $\mathcal{S}$ with the set of complements with respect to the new ground set, i.e., $\mathcal{S}'=\{(V\setminus S)\cup\{x\}\mid S\in\mathcal{S}\}$. Note that since $V\in \mathcal{S}$, we have $\{x\}\in\mathcal{S}'$. The new instance $(V',\mathcal{S}')$ has a hitting set of size $k$ if and only if the old one has one of size $k-1$. 

We can interpret the hitting set instance $(V',\mathcal{S}')$ as a set of vertices of $V'$ of the hypercube of dimension $d=\frac{|\mathcal{S}'|}{2}$. Every vertex $v\in V'$ has coordinates $v_i=\begin{cases}
                                                                                                                                                                                                    1, & v\in S_i,\\
-1, & v\in (V\setminus S_i)\cup\{x\}.\\                                                                                                                                                                                    
                                                                                                                                                                                                   \end{cases}$
A hitting set of $(V',\mathcal{S}')$ corresponds to a solution of \texttt{COORDINATE REVERSAL}, i.e., a minimum subset of $V'$ such that each coordinate is reversed, i.e., appears once positive and once negative. 

Conversely, an instance $(X,k)$ of \texttt{COORDINATE REVERSAL} in $Q_d$ is equivalent to the \texttt{HITTING SET} instance $(X,\mathcal{S})$, where $\mathcal{S}=\{X^+_e,X^-_e\mid e\in [d]\}$ and $X^{\pm}_e=\{x\in X\mid x_e=\pm\}$.%\qed
% 
% So, this is not precisely an L-reduction, but if $k-1\leq \log n$, then $k\leq \log 2n$. Since we want to reduce LOGproblems, this is OK if the size of the hitting set instance is measured in terms of the number of sets, since then the new instance is indeed of size $2n+2$.\textcolor{red}{not good enough for LOG problems, since here LOG refers to the size of V}
\end{proof}

Now, in~\cite{Alb-16} it is shown that in a partial cube $G=(V,E)\subseteq Q_d$ \texttt{HULL-NUMBER} coincides with \texttt{COORDINATE REVERSAL} for $V$ and $Q_d$, therefore, \texttt{HULL-NUMBER} in partial cubes is a special case of \texttt{COORDINATE REVERSAL}. In order to L-reduce \texttt{COORDINATE REVERSAL} to partial cube hull-number along the lines of Proposition~\ref{prop:COORDINATE}, it would be interesting to check, if given a subset $V'\subseteq Q_d$ a smallest partial cube containing $V'$ has to be polynomial in $|V'|+d$. Moreover it is important to maintain the same solution size with respect to \texttt{COORDINATE REVERSAL}. In~\cite[Theorem 15.59]{Flu-06} it is shown that {\texttt{LOGHITTING SET}} is LOGSNP (aka LOG[2]) complete. Hence, this would show LOGSNP-completeness of \texttt{LOGHULL-NUMBER} for partial cubes, one instance of which is calculating the dimension of a poset given its linear extensions, see~\cite{Alb-16}. So as a first step we wonder:

\begin{quest}~\label{quest:isohullcube}
 Let $X$ be a set of vertices of the hypercube $Q_d$, does there exist an isometric subgraph $G$ of $Q_d$, containing $X$, such that $|G|$ is polynomial in $|X|+d$?
\end{quest}

Let $M_k$ a $(0,1)$-matrix whose columns are all the $(0,1)$-vectors of length $k$. Now, $X_k\subseteq Q_{2^k}$ is defined as the set of rows of $M_k$. We do not know the answer to Question~\ref{quest:isohullcube} for the set $X_k$.

These questions lead to the problem of computing a small isometric subgraph containing a given set of vertices, which is the subject of the next section.

\section{Isometric hull}\label{sec:iso}

We recall the definitions related to the isometric hull from the introduction. Let $G=(V,E)$ be a graph. For any $v,u \in V$, let $\dist_G(u,v)$ denote the {\it distance} between $u$ and $v$, i.e., the minimum number of edges of a path between $u$ and $v$ in $G$. A subgraph $H=(V',E')$ (i.e., $V' \subseteq V$ and $E' \subseteq (V'\times V') \cap E)$) of $G$ is {\it isometric} if $\dist_G(u,v)=\dist_H(u,v)$ for any $u,v \in V'$.
Given $S \subseteq V$, an {\it isometric hull} of $S$ is any subgraph $H=(V',E')$ of $G$ such that $S \subseteq V'$ and $H$ is isometric. An isometric hull $H=(V',E')$ of $S$ is {\it minimum} if $|V'|$ is minimum, i.e., there are no isometric hulls of $S$ with strictly less vertices. %(All results in this note also hold if a minimum isometric hull is defined as minimum-inclusion subgraph.) 
Note that a set $S$ may have several minimum isometric hulls. As an example, consider the $4$-node cycle $C_4=(a,b,c,d)$: the subgraphs induced by $\{a,b,c\}$ and $\{a,d,c\}$ are minimum isometric hulls of $S=\{a,c\}$. More generally, for any $S=\{u,v\}$, inclusion-minimal isometric hulls of $S$ are any shortest path between $u$ and $v$. So, if $|S|=2$, all minimal isometric supergraphs of $S$ are of the same size and computing a minimum isometric hull of $S$ is easy. For $|S|>2$ it is easy to find examples with minimal isometric supergraphs that are not minimum. We show below that computing a minimum isometric hull is NP-complete if $|S|>2$.
An {\it isometric-hull set} (or simply {\it hull set}) $S \subseteq V$ of $G$ is any subset of the vertices such that $G$ is the (unique) minimum isometric hull of $S$. %Note that, by minimality of the size of a minimum isometric hull, if $S$ is an isometric-hull set of $G$, then $G$ is the unique isometric hull of $S$.  %Let $ih(G)$ be the minimum-size of an isometric-hull set of $G$. 

This section is devoted to prove the following theorems. 
%Theorem~\ref{theo:minIsomHull} states that computing a minimum isometric hull is NP-complete. Theorem~\ref{theo:minHullSet} states that computing a minimum-size hull set is complete for the second level of the Polynomial Hierarchy.

\begin{thm}\label{theo:minIsomHull}
Given an $n$-node bipartite graph $G=(V,E)$, $S \subseteq V$ and $k\in \mathbb{N}$, deciding whether there exists an isometric hull $H$ of $S$ with $|V(H)|\leq k$ is NP-complete, even if $|S|=3$ and $k=n-1$.
In particular, deciding whether a set $S$ of vertices is a hull set of a graph is coNP-complete.%, even if $|S|=3$. 
\end{thm}

%\begin{thm}\label{theo:HullSet}
%Given a graph $G=(V,E)$ and $S \subseteq V$, deciding whether the (unique) isometric hull of $S$ is $G$ is coNP-complete, even if $|S|=3$.
%\end{thm}

\begin{thm}\label{theo:minHullSet}
Given a graph $G$ and $k \in \mathbb{N}$, the problem of deciding whether $G$ admits an isometric-hull set of size at most $k$ is $\Sigma_2$-complete.
\end{thm}

%\smallskip

%\subsection{Minimum Isometric hull: no constraints on $|S|$ nor $k$}~\\

%\smallskip
Let us start with an easier result  where neither the size of the input set $S$ nor the size $k$ of the isometric hull are constrained, but where the class of bipartite graphs is restricted to have diameter $3$. 

\begin{lem}\label{lem:minIsomHull}
Given $G=(V,E)$ bipartite with diameter $3$, $S \subseteq V$ and $k\in \mathbb{N}$, deciding if there exists an isometric hull $I$ of $S$ with $|V(I)|\leq k$ is NP-complete.
\end{lem}

%\begin{proof}
%\begin{thm}
%Given a graph $G=(V,E)$, a set $X \subset V$ and an integer $k\in \mathbb{N}$ as inputs, the problem of deciding whether $X$ admits an isometric hull of size at most $k$ in $G$ is NP-hard.
%\end{thm}
\begin{proof}
The problem is clearly in NP since testing whether a subgraph is isometric can be done in polynomial-time. 

To prove that the problem is NP-hard, let us present a reduction from the \texttt{\texttt{HITTING SET}} Problem that takes a ground set $U = \{u_1,\cdots,u_n\}$ and a set ${\mathcal X} = \{X_1,\cdots,X_m\} \subseteq 2^U$ of subsets of $U$ and an integer $k$ as inputs and aims at deciding if there exists $K \subseteq U$ of size at most $k$ such that $K \cap X_j \neq \emptyset$ for every $j\leq m$. 
Note that we may assume that at least two sets of $\mathcal S$ are disjoint (up to adding a dummy vertex in $U$ and a set restricted to this vertex). 

Let us build the graph $G$ as follows. We start with the incidence graph of $(U,{\mathcal X})$, i.e., the graph with vertices $U \cup {\mathcal X}=\{u_1,\cdots,u_n,X_1,\cdots,X_m\}$ and edges $\{u_i,X_j\}$ for every $i \leq n$, $j\leq m$ such that $u_i \in X_j$. Then add a vertex $x$ adjacent to every vertex in $U$ and a vertex $y$ adjacent to every vertex in $\mathcal X$. Note that $G$ has diameter $3$. Finally, let $S= \{x\} \cup {\mathcal X}$. 

We show that $(U,{\mathcal X})$ admits a hitting set of size $k$ if and only if $S$ has an isometric hull of size $k+m+2$. Note that, because at least two sets are disjoint, $y$ must be in any isometric hull of $S$ (to ensure that these sets are at distance two). Moreover, for every set containing (at least) $x$, $y$ and $\mathcal X$, all distances are preserved but possibly the ones between $x$ and some vertices of $\mathcal X$. We show that  $I$ is an isometric hull of $S$ if and only if $K=V(I) \setminus (S \cup \{y\})$ is a hitting set of $(U,{\mathcal X})$. Indeed, for every $j\leq m$, the distance between $X_j$ and $x$ equals $2$ in $I$ if and only if $K$ contains a vertex $u_i$ adjacent to $X_j$, i.e., $K \cap X_j \neq \emptyset$ for every $j\leq m$. %\qed
\end{proof}

%\smallskip
%\subsection{Minimum Isometric hull: proof of Theorem~\ref{theo:minIsomHull}}~\\
%\subsection{Minimum Isometric hull (proof of Theorem~\ref{theo:minIsomHull})}
%\smallskip

Now, let us consider a restriction of Theorem~\ref{theo:minIsomHull} in the case $k=n-1$ (without constraint on $|S|$). For this purpose, we present a reduction from 3-SAT.

%In this section, we first consider the case $k=n-1$ (Lemma~\ref{lem2:minIsomHull}). For this purpose, we first give an alternative (and much more complicated proof) of Theorem~\ref{theo:minIsomHull} when neither $|S|$ nor $k$ are constrained. This is obtained by a reduction from 3-SAT. Then, the graph built from 3-SAT is slightly modified to handle the case $k=n-1$.

%\smallskip

\begin{figure}[h!]
\begin{center}
\scalebox{0.4}{\includegraphics{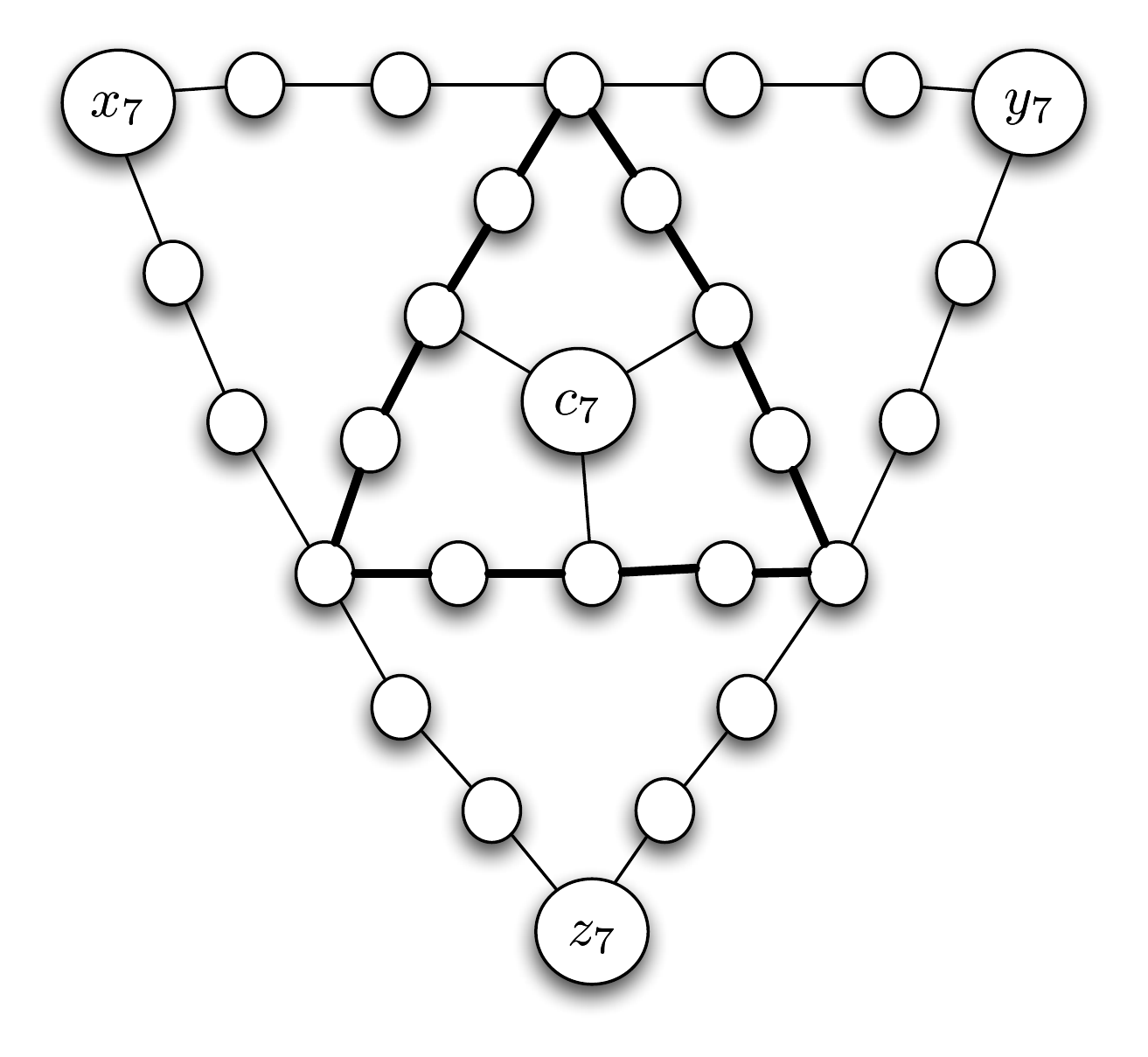}}
\end{center}
\caption{\label{fig-triangle}Example of $T_7$. Edges are bold only to better distinguish the different ``levels".}
\end{figure} 

\paragraph{\bf Preliminaries: the triangle gadget $T_{\gamma}$.}
Let us first describe a gadget subgraph, parameterized by an odd integer $\gamma$~\footnote{$\gamma$ is set odd only to avoid parity technicality in the computation of the distances.}, for which only $3$ vertices generate the whole graph. That is, we describe a graph $T_{\gamma}$ with size $\Theta(\gamma^2)$ such that there are $3$ vertices  (called the corners) whose minimum isometric hull is the whole graph. Moreover, some vertex (called the center)  of $T_{\gamma}$ is ``far" (at distance $\Theta(\gamma^2)$) from the corners.

Let $\gamma \in \mathbb{N}^*$ be any odd integer. Let us define recursively a $\gamma$-triangle with {\it corners} $\{x_{\gamma},y_{\gamma},z_{\gamma}\}$ and {\it center} $c_{\gamma}$ as follows.

A $3$-triangle $T_3$ is a $K_{1,3}$ where the big bipartition class $\{x_3,y_3,z_3\}$ are the corners and the center is the remaining vertex $c_3$.

Let $\gamma> 3$ be an odd integer and let $T_{\gamma-2}$ be a $(\gamma-2)$-triangle with corners $\{x_{\gamma-2},y_{\gamma-2},z_{\gamma-2}\}$ and center $c_{\gamma-2}$. The $\gamma$-triangle $T_{\gamma}$ is obtained as follows. First, let $U_{\gamma}$ be the cycle of length $3(\gamma-1)$ with vertices $x_{\gamma},y_{\gamma},z_{\gamma}$ that are pairwise at distance $\gamma-1$. For any $u,v \in \{x_{\gamma},y_{\gamma},z_{\gamma}\}$, let $a_{uv}$ be the vertex at distance $\lfloor \gamma/2 \rfloor$ from $u$ and $v$ in $U_{\gamma}$. The graph $T_{\gamma}$ is obtained from $U_{\gamma}$ and $T_{\gamma-2}$ by identifying $x_{\gamma-2},y_{\gamma-2},z_{\gamma-2}$ with $a_{x_{\gamma},y_{\gamma}}$, $a_{y_{\gamma},z_{\gamma}}$ and $a_{z_{\gamma},x_{\gamma}}$, respectively. The corners of $T_{\gamma}$ are $x_{\gamma},y_{\gamma}$ and $z_
\gamma$, and the center $c_{\gamma}$ of $T_{\gamma}$ is the center $c_{\gamma-2}$ of $T_{\gamma-2}$. Note that the center $c_{\gamma}$ of $T_{\gamma}$ is the center $c_3$ of the ``initial" ``triangle" $T_3$. An example is depicted on Figure~\ref{fig-triangle}.

The following claim can be easily proved by induction on $\gamma$. The second statement also comes from the fact that $T_{\gamma-2}$ is an isometric subgraph of $T_{\gamma}$. 

\begin{claim}
For any odd integer $\gamma>3$, let $T_{\gamma}$ with corners $S=\{x_{\gamma},y_{\gamma},z_{\gamma}\}$
\begin{itemize}
\item $|V(T_{\gamma})|= |V(T_{\gamma-2})|+ 3(\gamma-2) = \Theta(\gamma^2)$;
\item the (unique) isometric hull of $S$ is $T_{\gamma}$;
\item the distance between any two corners in $T_{\gamma}$ is $\gamma-1$;
\item the distance between the center and any corner in $T_{\gamma}$ is $\sum_{i=1}^{\lceil\frac{\gamma}{2}\rceil}i= \Theta(\gamma^2)$;
\item since $T_{\gamma}$ is planar and all faces are even, $T_{\gamma}$ is bipartite. %Furthermore, all corners in the same part of the bi-partition.
\end{itemize}
\end{claim}

\begin{lem}\label{lem2:minIsomHull}
Given a bipartite $n$-node graph $G=(V,E)$, $X \subseteq V$, deciding whether there exists an isometric hull $I$ of $X$ with $|V(I)|<n$ is NP-complete.
\end{lem}
\begin{proof}
The problem is clearly in NP since testing whether a subgraph is isometric can be done in polynomial-time. To prove that the problem is NP-hard, let us present a reduction from $3$-SAT. 

Let $\Phi$ be a CNF formula with $n$ variables $v_1,\cdots,v_n$ and $m$ clauses $C_1,\cdots,C_m$. Let us describe a graph $G_0=(V,E)$, $S\subseteq V$ and $k\in \mathbb{N}$ such that an isometric hull of $S$ has size at most $k$ if and only if $\Phi$ is satisfiable.

\noindent Let $\alpha,\beta$ and $\gamma$ be three integers satisfying: $\alpha$ and $ \beta$ are even and $\gamma$ is odd and
$$m<<2\alpha<2\beta<\gamma<2(\alpha+\beta).$$

The graph $G_0$ is built by combining some variable-gadgets, clause-gadgets and adding some paths connecting the variable-gadgets with some particular vertex $r$.

\smallskip

\noindent{\bf Variable-gadget.} For any $1\leq i \leq n$, the variable-gadget $V^i$ consists of a cycle of length $4 \alpha$ with four particular vertices $d_i,n_i,p_i,g_i$ such that $d_i$ and $g_i$ are antipodal, i.e., at distance $2 \alpha$ of each other, $n_i$ and $p_i$ are antipodal, and $\dist_{V^i}(d_i,n_i)=\dist_{V^i}(d_i,p_i)=\dist_{V^i}(g_i,n_i)=\dist_{V^i}(g_i,p_i)=\alpha$. Let $P^i$ and $N^i$ be the shortest path between $d_i$ and $g_i$ in $V^i$ passing through $p_i$ and $n_i$, respectively. 

%\smallskip

\noindent{\bf Clause-gadget.} For any $1\leq j \leq m$ and clause $C_j=(\ell_i \wedge \ell_k \vee \ell_h)$ (where $\ell_i$ is the literal corresponding to variable $v_i$ in clause $C_j$), the clause-gadget $C^j$ is a $\gamma$-triangle with corners denoted by $\ell_i,\ell_k,\ell_h$ (abusing the notation, we identify the corner-vertices and the literals they correspond to) and center denoted by $c^j$.

%\smallskip

\noindent{\bf The graph $G_0$.} The graph $G_0$ is obtained as follows. First, let us start with disjoint copies of $V^i$, for $1\leq i \leq n$, and of $C^j$, for $1 \leq j \leq m$. Then, add one vertex $r$ and, for any $1\leq i \leq n$, add a path $P(r,d_i)$ of length $\beta$ between $r$ and $d_i$ and a path $P(r,g_i)$ of length $\beta$ between $r$ and $g_i$ (these $2n$ paths are vertex-disjoint except in $r$). Finally, for any $1\leq j \leq m$ and any literal $\ell_i$ in the clause $C_j$, let us identify the corner $\ell_i$ of $C^j$ with vertex $p_i$ (in the variable-gadget $V^i$) if variable $v_i$ appears negatively in $C_j$ (i.e., if $\ell_i=\bar{v}_i$) and identify the corner $\ell_i$ of $C^j$ with vertex $n_i$ if variable $v_i$ appears positively in $C_j$ (i.e., if $\ell_i=v_i$). Let us emphasize that, if variable $v_i$ appears positively (negatively) in $C_j$, then a corner of $C^j$ is identified with a vertex of the path $N^i$ ($P^i$). By the parity of $\alpha,\beta$ and $\gamma$, $G_0$ is clearly bipartite.
An example is depicted in Figure~\ref{fig-reduction}.

\begin{figure}[t!]
\begin{center}
\scalebox{0.4}{\includegraphics{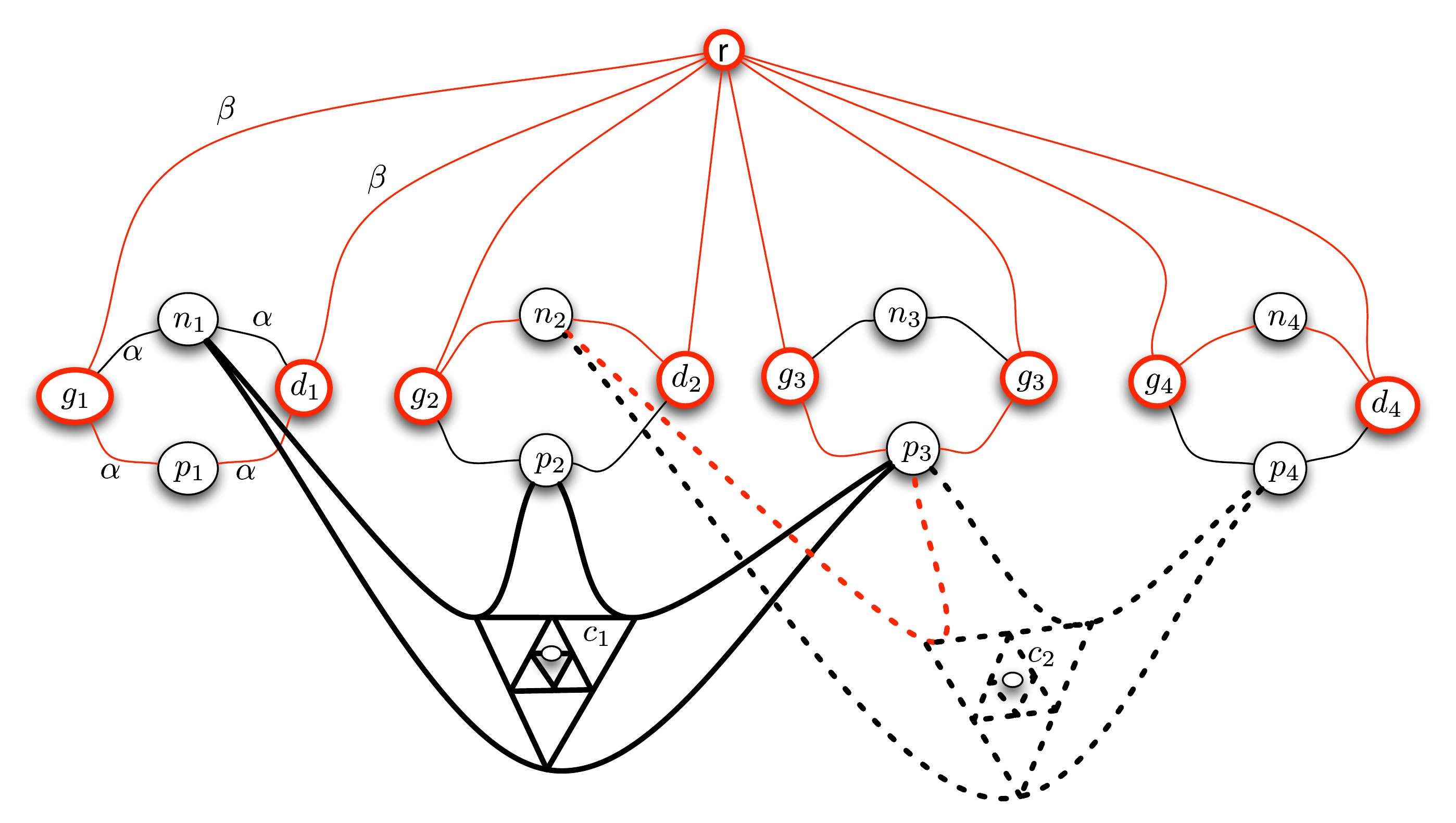}}
\end{center}
\caption{\label{fig-reduction} An example for the graph $G_0$ of the reduction of Lemma~\ref{lem:minIsomHull} for $\Phi$ with variables $v_1,v_2,v_3,v_4$ and two clauses $C_1=v_1 \vee \bar{v}_2 \vee \bar{v}_3$ and $C_2=v_2 \vee \bar{v}_3 \vee \bar{v}_4$. The solid bold lines represent the clause-gadget $C^1$ which is a $\gamma$-triangle (only few ``levels" are depicted), and the dotted bold lines represent the clause-gadget $C^2$. The vertices $c_1$ and $c_2$ denote the centers of $C^1$ and $C^2$ respectively. Red vertices are the ones of the set $S=\{r,d_1,g_1,d_2,g_2,d_3,g_3,d_4,g_4\}$. Finally, the red subgraph is the isometric hull of $S$ corresponding to the truth assignment $(v_1,v_2,v_3,v_4)=(1,0,1,0)$.
The graph $G$ is obtained from $G_0$ by adding a vertex $q$ adjacent to $c_1$ and $c_2$.}
\end{figure} 

\smallskip

\noindent{\bf The set $S$.} Finally, let $S= \{r\} \cup \{d_i,g_i \mid 1 \leq i \leq n\}$. 

%\smallskip

We first show that $S$ has an isometric hull of size at most $k:=n (\alpha + 2 \beta)+m\gamma$ in $G_0$ if and only if $\Phi$ is satisfiable. %Due to lack of space, the proof of the following claim  is postponed in the Appendix~\ref{App:1}.
% % 
% % \begin{claim}
% % $S$ has an isometric hull of size at most $k:=n (\alpha + 2 \beta)+m\gamma$ in $G_0$ if and only if $\Phi$ is satisfiable. 
% % \end{claim}

\begin{claim}
$S$ has an isometric hull of size at most $k:=n (\alpha + 2 \beta)+m\gamma$ in $G_0$ if and only if $\Phi$ is satisfiable. 
\end{claim}
\begin{proofclaim}

%\smallskip

Let us start with some simple observations (following from the constraints on $\alpha,\beta$ and $\gamma$):
\begin{enumerate}
\item For any $1\leq i \leq n$, $\dist_G(d_i,g_i)=\dist_G(p_i,n_i)=2 \alpha$ and there are exactly two shortest paths $P^i$ and $N^i$ between $d_i$ and $g_i$. Intuitively, choosing $P^i$ (resp., $N^i$) in the isometric hull will correspond to a positive (resp., negative) assignment of variable $v_i$.
\item For any $1\leq i \leq n$, $\dist_G(r,d_i)=\dist_G(r,g_i)=\beta$ and $P(r,d_i)$ (resp., $P(r,g_i)$) is the unique shortest path between $r$ and $d_i$ (resp., between $r$ and $g_i$). In particular, each of these paths has to be in any isometric hull of $S$.
\item For any $1\leq i<j \leq n$, $\dist_G(d_j,d_i)=\dist_G(d_j,g_i)=2\beta$ and the unique shortest path between them is the one going through $r$ (because $2\beta<\gamma$).
\item For any $1 \leq j \leq m$ and clause  $C_j=(\ell_i \vee \ell_k \vee \ell_h)$, $\dist_G(\bar{\ell}_i,\bar{\ell}_h)=\dist_G(\bar{\ell}_h,\bar{\ell}_k)=\dist_G(\bar{\ell}_i,\bar{\ell}_k)=\gamma-1$ (where $\bar{\ell}_i$ denotes $n_i$ if $v_i$ appears positively in $C_j$ and it denotes $p_i$ otherwise). This is because $\gamma < 2(\alpha + \beta)$ and the unique shortest path between these vertices is the one in $C^j$.
\item  For any $1\leq h<k\leq n$, $\ell_h \in \{n_h,p_h\}$ and $\ell_k \in  \{n_k,p_k\}$ such that literals $\bar{\ell}_h$ and $\bar{\ell}_k$ do not appear in a same clause, then $\dist_G(\ell_h, \ell_k) = 2(\alpha + \beta)$ (because $2 \beta < \gamma$). In particular, every shortest path between $\ell_h$ and $\ell_k$ does not cross any clause-gadget.
\item Let $1\leq h<k\leq n$, $\ell_h \in \{n_h,p_h\}$ and $\ell_k \in  \{n_k,p_k\}$ appearing in a clause $C_j$. For any vertex $u$ in the shortest path between $\ell_h$ and $\ell_k$ in the clause-gadget $C^j$, and for any $v \in  \{d_i,g_i\}$ for some $i\notin \{h,k\}$, $\dist_G(u,v) \leq \gamma/2+\alpha+2 \beta$. In particular, any shortest path between $u$ and $v$ does not pass through the third corner (different from $\ell_h$ and $\ell_k$) of $C^j$. This is because $ \gamma  > 2 \beta$. 
\end{enumerate}

%\begin{claim}
%$S$ has an isometric hull of size at most $k:=n (\alpha + 2 \beta)+m\gamma$ in $G_0$ if and only if $\Phi$ is satisfiable. 
%\end{claim}
%\begin{proofclaim}
\begin{itemize}
\item
First, let us show that, if $\Phi$ is satisfiable, there is an isometric hull of $S$ with size at most $n (\alpha + 2 \beta)+m\gamma$ in $G_0$.
Indeed, consider a truth assignment of $\Phi$ and let $H$ be the subgraph defined as follows. For any $1\leq i \leq n$, the paths $P(r,d_i)$ and $P(r,g_i)$ belong to $H$. For any  $1\leq i \leq n$, if $v_i$ is assigned to True, add $P^i$ in $H$, and add $N^i$ otherwise. Finally, for any $1\leq j\leq m$, for any two corners of the clause-gadget $C^j$, if these two corners are in $H$, then add to $H$ the path of length $\gamma-1$ between them in $C^j$.

Clearly, $H$ contains all vertices in $S$. To show that $H$ is isometric, let us first show that any clause-gadget has at most two corners in $H$. Let $x \in \{n_i,p_i\}$ be a corner of a clause-gadget $C^j$ which is in $H$. If $x=n_i$ (resp., $x=p_i$) is in $H$, it implies that the path $N^i$ (resp., $P^i$) has been added in $H$. Therefore, the variable $v_i$ is assigned to False (resp., to True) in the assignment. On the other hand, if $x=n_i$ (resp., $x=p_i$) is a corner of $C^j$, it means that the variable $v_i$ appears positively (resp., negatively) in clause $C_j$. Altogether, this implies that, in the assignment, Variable $v_i$ cannot satisfy clause $C_j$. Since the assignment satisfies $\Phi$, each clause must be satisfied by at least one of its variables, which implies that at least one of its corners is not in $H$. 

To sum-up $H$ consists of the $2n$ paths from $r$ to the vertices $d_i,g_i$, $1\leq i \leq m$, of exactly on path $P^i$ or $N^i$, $1\leq i \leq m$, and of at most one path between two corners of $C^j$, $1\leq j \leq m$. Hence, $H$ has at most $n (\alpha + 2 \beta)+m\gamma$ vertices. The fact that $H$ is isometric comes from the above observations on the shortest paths in $G_0$.
\item
To conclude, let us show that, if $\Phi$ is not satisfiable, then any isometric hull of $S$, in $G_0$, has size at least $n (\alpha + 2 \beta)+\Omega(\gamma^2)$, i.e., strictly larger than  $n (\alpha + 2 \beta)+m\gamma$ (since $\gamma >>m$). 

As already mentioned, any isometric hull of $S$ has to contain each of the paths $P(r,d_i)$ and $P(r,g_i)$ and at least one of the paths $P^i$ and $N^i$, $1\leq i \leq n$. This consists of at least $n (\alpha + 2 \beta)$ vertices. It remains to show that, for any isometric hull $H$ of $S$, there exists $j\leq m$ such that the entire clause-gadget $C^j$ belongs to $H$. This will consist of $\Omega(\gamma^2)$ additional vertices. 

Let $H$ be an isometric hull of $S$. For any $1 \leq i \leq n$, at least $P^i$ or $N^i$ belongs to $H$. If $P^i$ belongs to $H$, assign variable $v_i$ to True and assign it to False otherwise. Since $\Phi$ is not satisfiable, there is a clause $C_j=(\ell_i \vee \ell_k \vee \ell_h)$ that is not satisfied. Let $u \in \{i,h,k\}$. If $v_u$ appears positively (resp., negatively) in $C_j$, then $v_u$ is assigned to False (resp., to True) since $C_j$ is not satisfied. Moreover, it implies that $P^u$ (resp., $N^u$) belongs to $H$. By construction, the corner $\ell_u$ of $C^j$ belongs to $P^u$ (resp., $N^u$) and so, $\ell_u$ belongs to $H$. Hence, all the three corners of $C^j$ belong to $H$ and it is easy to see that the entire $C^j$ must belong to $H$ since, recursively, all paths in $C^j$ have to be added to preserve the fact that $H$ is isometric.
\end{itemize}
\end{proofclaim}

\noindent To prove Lemma~\ref{lem2:minIsomHull}, from $G_0$, let us build a graph $G$ such that $S$ (which remains unchanged) has an isometric hull of size at most $|V(G)|-1$ if and only if $\Phi$ is satisfiable. The graph $G$ is obtained from $G_0$ by adding to it a gadget (one vertex) that will ensure that if the center of one clause-gadget belongs to an isometric hull (recall that, in the first part of the proof, this is the case if and only if $\Phi$ is not satisfiable), then all vertices of the graph will have to be in the isometric hull. 

Let us add to $G_0$ one vertex $q$ adjacent to all the centers of the clause-gadgets in $G_0$. Note that the obtained graph $G$ is bipartite. 
%inally, for any $1 \leq i \leq n$, add a path of length $\delta$ between $q$ and $p_i$, and a path of length $\delta$ between $q$ and $n_i$. All these paths are vertex-disjoint except in $q$.  
\begin{itemize}
\item
If $\Phi$ is satisfiable, consider any truth assignment and let $H$ be the subgraph (as defined in the previous proof) that consists of the $2n$ paths from $r$ to the vertices $d_i,g_i$, $1\leq i \leq m$, of exactly on path $P^i$ or $N^i$, $1\leq i \leq m$, and of at most one path between two corners of $C^j$, $1\leq j \leq m$. Because each center of a clause-gadget is at distance $\Omega(\gamma^2)$ from any vertex of $H$, the addition of vertex $q$ in the graph has not modified the distances between vertices in $H$. Therefore, $H$ is isometric (as in the first part of the proof) and $S$ is not a hull set of $G$.
\item
If $\Phi$ is not satisfiable, then we prove that $S$ is a hull set of $G$. 

If $\Phi$ is not satisfiable, then it can be shown as previously that any isometric hull $H$ of $S$ contains at least one of the path $P^i$ or $N^i$ for any $1\leq i \leq n$, and an entire clause-gadget $C^j$ for some $j\leq m$. In particular, the center $c$ of $C^j$ belongs to $H$. Now, let $1\leq i \leq n$ and let us assume that $N^i$ is not in $H$. Note that, in this case, $P^i$ must be in $H$. Therefore $p_i \in V(H)$ and $n_i \notin V(H)$. By assumption, there is a clause $C_z$ that contains $v_i$ positively and does not contain $v_i$ negatively. By construction, the clause-gadget $C^z$ has $n_i$ as a corner and $p_i$ is not a corner of $C^z$. Note that $z\neq j$ since all corners of $C^j$ belong to $H$. Now, any shortest path between $c$ and $d_i$ must go from $c$ to $q$ then to the center of the clause-gadget $C_z$ and then through $n_i$ to $d_i$. In particular, $n_i$ must be added to the isometric hull. It can be proved similarly that if $P^i$ does not belong to $H$, then $p_i$ has to be included into $H$. 

Altogether, we just proved that, for any $1\leq i \leq n$, both $n_i$ and $p_i$ belong to $H$. It is easy to conclude that $H=G$. Indeed, in particular, any clause-gadget has all its corner in $H$ and therefore, the entire clause-gadget must be included in $H$.%\qed
\end{itemize}

\end{proof}

Finally, to prove Th.~\ref{theo:minIsomHull}, we will reduce the problems that we proved NP-complete in Lemma~\ref{lem2:minIsomHull} to the same problems in the case $|S|=3$. Note that, in both reductions of Lemmas~\ref{lem:minIsomHull} and~\ref{lem2:minIsomHull}, the distance between any pair of vertices of $S$ is even, so both problems are NP-complete with this extra constraint. 

\begin{proof}{\bf of Th.~\ref{theo:minIsomHull}.}
Let $G,S,k$ be an instance of the problem of finding an isometric hull of $S$ with size at most $k$. Let $n=|V(G)|$ and let $S=\{u_1,\cdots,u_s\}$. Moreover, let us assume that the distance between any pair of vertices of $S$ is even. 

Let $G'$ be obtained as follows. Start with a copy of $G$, a path $P=(x=v_0,v_1,w_1,v_2,w_2,\cdots$ $,w_{s-1},v_s,$ $v_{s+1}=y\}$ and a vertex $z$. Let $n'=n$ if $n$ even and $n'=n+1$ otherwise. For any $1\leq i \leq s$, add a path of length $n'$ between $v_i$ and $u_i$ and add a path of length $n'$ between $z$ and $u_i$. Note that $G$ is an isometric subgraph of $G'$ and that $G'$ is bipartite.  Finally, let $S'=\{x,y,z\}$. 

Any isometric hull $H$ of $S'$ has to contain the (unique) shortest path $P$ between $x$ and $y$. Hence, for any $1 \leq i \leq s$, $H$ contains $v_i$ and therefore must contain the (unique) shortest path $P_i$ between $v_i$ and $z$ (of length $2n'$). In particular $H$ contains $u_i$ for any $1 \leq i \leq s$. Since $G$ is isometric in $G'$, then the subgraph induced by the vertices in $V(G) \cap V(H)$ is an isometric hull of $S$ in $G$. 

Therefore, $S$ admits an isometric hull of size at most $k$ in $G$ if and only if $S'$ admits an isometric hull of size $k+|V(G')\setminus V(G)|= k+2sn'+s+1$. In particular, if $k=n-1$, then the formula gives $|V(G')|-1$.%\qed
\end{proof}

Note that deciding whether a set $S$ of vertices is not a hull set of an $n$-node graph is equivalent to decide whether $S$ has an isometric hull of size $<n$. Therefore:
\begin{cor}
Deciding whether a set of vertices is a hull set is coNP-complete.
\end{cor}

%\subsection{Minimum Hull Set (proof of Theorem~\ref{theo:minHullSet})}

Finally, to prove Theorem~\ref{theo:minHullSet}, we present a reduction from the problem of satisfiability for quantified Boolean formulas with $2$ alternations of quantifiers $QSAT_2$. The reduction is an adaptation of the one presented in the proof of Theorem~\ref{theo:minIsomHull}. %Due to lack of space, the proof is postponed to Appendix~\ref{App:3}.%, while there are some necessary differences. 

In the proof below, we will use the following easy claim to force some vertices to belong to any hull set.

\begin{claim}
For any graph $G=(V,E)$ and any vertex $v\in V$ such that $G \setminus v$ is isometric, we have that $v$ has to belong to any hull set of $G$. 
In particular, any one-degree vertex of $G$ has to belong to any hull set of $G$.
\end{claim}

\begin{proof}{\bf of Theorem~\ref{theo:minHullSet}.}
First, the problem is in $\Sigma_2$. Indeed, by Theorem~\ref{theo:minIsomHull}, a certificate $S$ (i.e., a set of vertices which is supposed to be a hull set of $G$) can be checked using an NP oracle. 

To prove that it is hard for $\Sigma_2$, let us give a reduction from $QSAT_2$ where the input is a Boolean formula $\Phi$ on two sets $X=\{x_1,\cdots,x_{n_x}\}$ and $Y=\{y_1,\cdots,y_{n_y}\}$ of variables and the question is to decide whether $\exists X, \forall Y, \Phi(X,Y)$. We moreover may assume that $\Phi$ is $3$-DNF formula, i.e., the disjunction of conjunctive clauses $C_1,\cdots,C_m$ with $3$ variables each. We also assume that, for each variable, some clause contains it positively and some clause contains it negatively, and that no variable appears positively and negatively in some clause.

Let us describe a graph $G=(V,E)$ and $k\in \mathbb{N}$ such that there exists a hull set $S$ of size at most $k$ if and only if $\exists X, \forall Y, \Phi(X,Y)$.

Let $\alpha,\beta$ and $\gamma$ be three integers satisfying: $\alpha$ and $ \beta$ are even and $\gamma$ is odd and
$$m<< 2\alpha<2\beta<\gamma<2(\alpha+\beta).$$

The graph $G$ is built by combining some variable-gadgets, clause-gadgets and adding some paths connecting the variable-gadgets with some particular vertex $r$. We emphasize the differences with the graph proposed in previous subsection. 

\smallskip

\noindent{\bf Variable-gadget.} For any $1\leq i \leq n_y$, the variable-gadget $Y^i$ consists of a cycle of length $4 \alpha$ with four particular vertices $d^y_i,n^y_i,p^y_i,g^y_i$ such that $d^y_i$ and $g^y_i$ are antipodal, i.e., at distance $2 \alpha$, $n^y_i$ and $p^y_i$ are antipodal, and $\dist_{Y^i}(d^y_i,n^y_i)=\dist_{Y^i}(d^y_i,p^y_i)=\dist_{Y^i}(g^y_i,n^y_i)=\dist_{Y^i}(g^y_i,p^y_i)=\alpha$.  Let $P^i_y$ (resp., $N^i_y$) be the shortest path between $d^y_i$ and $g^y_i$ in $Y^i$ passing through $p^y_i$ and $n^y_i$, respectively. 

Moreover, let us add a one-degree vertex $dd^y_i$ adjacent to $d^y_i$ and a one-degree vertex $gg^y_i$ adjacent to $g^y_i$ (This is the first difference with the previous section). By the above claim both vertices $dd^y_i$ and $gg^y_i$ have to belong to any hull set of $G$.

For any $1\leq i \leq n_x$, the variable-gadget $X^i$, the vertices $d^x_i,n^x_i,p^x_i,g^x_i,dd^x_i,gg^x_i$ and the paths $P^i_x$ and $N^i_x$ are defined similarly.

%\smallskip

\noindent{\bf Clause-gadget.} For any $1\leq j \leq m$ and clause $C_j=(\ell_i \wedge \ell_k \wedge \ell_h)$, the clause-gadget $C^j$ is a $\gamma$-triangle with corners denoted by $\ell_i,\ell_k,\ell_h$ (abusing the notation, we identify the corner-vertices and the literals they correspond to) and center denoted by $c^j$.

%\smallskip

\noindent{\bf The graph $G$.} The graph $G$ is obtained as follows. First, let us start with disjoint copies of $X^i$, for $1\leq i \leq n_x$, of $Y^i$ for $1\leq i \leq n_y$, and of $C^j$, for $1 \leq j \leq m$. Then, add one vertex $r$ and, for any $1\leq i \leq n_x$, add a path $P(r,d^x_i)$ of length $\beta$ between $r$ and $d^x_i$ and a path $P(r,g^x_i)$ of length $\beta$ between $r$ and $g^x_i$ (these $2n_x$ paths are vertex-disjoint except in $r$). Similarly, for any $1\leq i \leq n_y$, add a path $P(r,d^y_i)$ of length $\beta$ between $r$ and $d^y_i$ and a path $P(r,g^y_i)$ of length $\beta$ between $r$ and $g^y_i$ (these $2n_y$ paths are vertex-disjoint except in $r$).

Then, add a one-degree vertex $r'$ adjacent to $r$ (This is another difference with the previous section). Again, by the above claim, vertex $r'$ has to belong to any hull set of $G$.

A main difference with the construction in the previous section is the way the clause-gadgets are connected to the variable-gadgets. Intuitively, this is because we consider now a DNF formula while previously it was a CNF formula. 

For any $1\leq j \leq m$ and any literal $\ell_i$ in the clause $C_j$ (corresponding to some variable $v_i \in X\cup Y$), let us identify the corner $\ell_i$ of $C^j$ with vertex $p_i$ (in the variable-gadget of variable $v_i$) if variable $v_i$ appears positively in $C_j$ and identify the corner $\ell_i$ of $C^j$ with vertex $n_i$ if variable $v_i$ appears negatively in $C_j$. Let us emphasis that, contrary to the previous section, if variable $v_i$ appears positively (resp., negatively) in $C_j$, then a corner of $C^j$ is identified with a vertex of the path $P^i$ (resp., $N^i$). 

Finally, add a vertex $q$ adjacent to all centers of the clause-gadgets.

%\smallskip

\noindent{\bf The last touch.} Let $\delta$ be any odd integer larger than the diameter of the graph built so far. For any $1\leq i \leq n_x$, let us add a path $H^i$ of length $\delta$ between $p^x_i$ and $n^x_i$. 

The key point is that any hull set of $G$ has to contain at least one internal vertex of each path $H^i$. Indeed, by the choice of $\delta$, for any $1\leq i \leq n_x$, the graph obtained from $G$ by removing the internal vertices of $H^i$ is isometric in $G$. 

Another important remark is that, since $\delta$ is odd, each vertex in $H^i$ is either closer to $p^x_i$ than to $n^x_i$ or vice-versa (no vertex is at equal distance from both). For any $1\leq i \leq n_x$, let $\{h^p_i,h^n_i\}$ be the middle edge of $H^i$ where $h^p_i$ is closer than $p^x_i$ and $h^n_i$ is closer than $n^x_i$

\smallskip

As we have already said, any hull set of $G$ must contain all vertices in $I=\{dd^x_i,gg^x_idd^y_j,gg^y_j \mid 1\leq i \leq n_x, 1\leq j \leq n_y\} \cup \{r'\}$ and at least one internal vertex in $H^i$ for each $1 \leq i \leq n_x$. That is, any hull set of $G$ has at least $3n_x+2n_y+1$ vertices.

We show that $G$ has a hull set of size $3n_x+2n_y+1$ if and only if $\exists X, \forall Y, \Phi(X,Y)$.

\begin{itemize}
\item 
First, assume that there exists an assignment $X^*$ of $X$ such that every assignment of $Y$ satisfies $\Phi(X,Y)$. For any $1\leq i \leq n_x$, let $s_i$ denote the vertex $h^p_i$ if variable $x_i$ is set to True, and $s_i$ denote $h^n_i$ otherwise.

We prove that $S=I \cup \{s_1,\cdots,s_{n_x}\}$ is a hull set of $G$, i.e., $G$ is the unique isometric hull of $S$. 

If $s_i = h^p_i$ then the path $P^i_x$ and the shortest path from $p^x_i$ to $h^p_i$ (i.e., the subpath of $H^i$) must belong to any isometric hull of $S$. Symmetrically, if $s_i = n^p_i$ then the path $N^i_x$ and the shortest path from $n^x_i$ to $h^n_i$ (i.e., the subpath of $H^i$) must belong to any isometric hull of $S$. 

Moreover, for any $1\leq i \leq n_y$, any isometric hull of $S$ must contain either $P^i_y$ or $N^i_y$. 

Let us consider any isometric hull $H$ of $S$ and, for any $1\leq i \leq n_y$, let $L_i \in \{P^i_y,N^i_y\}$ be a path contained in $H$.

Consider the assignment $Y^*$ of $Y$ defined by $H$ as follows: if $L_i=P^i_y$  then variable $y_i$ is set to true, and it is set to False otherwise (i.e., if $L_i=N^i_y$). Since the formula is true for any assignment of $Y$, then $\Phi(X^*,Y^*)$ is true. In particular, there is a clause $C_j$ satisfied by all its variables. By definition of $X^*,Y^*$ and $H$, this implies that all its three corners belong to $H$ and, as in the proof of Lemma~\ref{lem:minIsomHull}, this implies that the entire clause-gadget $C^j$ is in $H$. Therefore, using vertex $q$ as in proof of Lemma~\ref{lem2:minIsomHull}, this implies that all vertices $p^i_x, n^i_x$ for $1\leq i \leq n_x$ compared and all vertices $p^i_y, n^i_y$ for $1\leq i \leq n_y$ belong to $H$. From there, it is easy to conclude that all vertices of $G$ belong to $H$. Therefore, $G$ is the unique isometric hull of $S$ and $S$ is a hull set of the desired size.
\item To conclude, we prove that, if for any assignment $X^*$ of $X$ there exists an assignment $Y^*$ of $Y$ such that $\Phi(X^*,Y^*)$ is False, then no set of at most $3n_x+2n_y+1$ vertices is a hull set of $G$. 

Let $S$ be a set of at most $3n_x+2n_y+1$ vertices. As already said, to be a hull set, $S$ must be equal to $I \cup \{s_1,\cdots,s_{n_x}\}$ where, for any $1\leq i \leq n_x$, vertex $s_i$ is an internal vertex of the path $H^i$. 

Let $X^*$ be the assignment of $X$ defined as follows: for any $1\leq i \leq n_x$, variable $x_i$ is set to True if $s_i$ is closer to $h^p_i$ and $x_i$ is set to False otherwise.

By assumption, there is an assignment $Y^*$ of $Y$ such that $\Phi(X^*,Y^*)$ is False.

Let $H$ be the subgraph of $G$ built as follows. First, $H$ contains $S$ and all paths $P(r,d_i^x)$ and $P(r,g^i_x)$ for $1\leq i \leq n_x$ and $H$ contains all paths $P(r,d_i^y)$ and $P(r,g^i_y)$ for $1\leq i \leq n_y$. For any $1\leq i \leq n_x$, $H$ contains $P^i_x$ and the shortest path between $s_i$ and $p^i_x$ if $x_i$ is assigned to True, and $H$ contains $N^i_x$ and the shortest path between $s_i$ and $n^i_x$ if $x_i$ is assigned to  False. For any $1\leq i \leq n_y$, $H$ contains $P^i_y$  if $y_i$ is assigned to True, and $H$ contains $N^i_y$ if $y_i$ is assigned to False.

As in the proof of Lemma~\ref{lem:minIsomHull}, because no clause is satisfied by $X^*\cup Y^*$, it can be proved that each clause-gadget has at most two corners in the current graph $H$. 

Finally, for any clause-gadget $C^j$ that has exactly to corners in $H$, add to $H$ the shortest path (in $C^j$) between these two corners. 

Similar arguments as those in the proof of Lemma~\ref{lem:minIsomHull} give that $H$ is a proper isometric subgraph of $G$ and contains $S$. Therefore, $S$ is not a hull set of $G$. %\qed
\end{itemize}
\end{proof}

\section{Further work}
We have devised a polytime algorithm for MGS when all images of a pseudo-closure are given as an input. While pseudo-closures generalize closures, they do not capture other generalizations from the literature such as preclosures~\cite{Ark-90} (since they are not idempotent) or closure functions of greedoids~\cite{Bjo-92} (since they are extensive). Can similar algorithms be provided for these classes?

An open problem with respect to closures is Question~\ref{quest:graphiclattices}, i.e., find a characterization of those closures coming from the convex subgraphs of a graph. The corresponding question for (finite) metric spaces is also open, see~\cite{Her-94}.
Moreover, we wonder about the complexity of \texttt{LOGHULL-NUMBER}, even for partial cubes. A particular question arising in this context is, whether \texttt{HULL-NUMBER} admits an FPT algorithm parametrized by solution size $k$.
Finally, we would like to recall Question~\ref{quest:isohullcube}, i.e., is there a subset $X$ of the hypercube $Q_d$ such that a smallest partial cube in $Q_d$ containing $X$ is not of polynomial size in $d+|X|$?

\bibliographystyle{plain}
\bibliography{latthull}

\end{document}